\documentclass[11pt]{article}

\usepackage{amsmath}
\usepackage{amssymb}
\usepackage{amsthm}
\usepackage{graphicx}
\usepackage{amscd}
\usepackage{epic, eepic}
\usepackage{url}
\usepackage{xcolor}
\usepackage{todonotes}
\usepackage{enumerate}
\usepackage{comment}
\usepackage{hyperref}
\usepackage{bbold}
\usepackage[utf8]{inputenc}
\usepackage{tabularray}
\newtheorem*{theorem*}{Theorem}
\newtheorem{theorem}{\bf Theorem}[section]
\newtheorem{lemma}[theorem]{\bf Lemma}
\newtheorem{corollary}[theorem]{\bf Corollary}

\newtheorem{problem}[theorem]{\bf Problem}
\newtheorem{prop}[theorem]{\bf Proposition}

\theoremstyle{definition}

\newtheorem{example}[theorem]{\bf Example}

\newcommand{\F}{\mathbb{F}}
\newcommand{\Cay}{\mathrm{Cay}}
\newcommand{\Sym}{\mathrm{Sym}}
\newcommand{\Alt}{\mathrm{Alt}}
\newcommand{\supp}{\mathrm{supp}}
\title{Tilings of symmetric and alternating groups by conjugacy classes}
\author{G\'abor Somlai \thanks{School of Mathematics and Statistics, The University of Melbourne, Parkville, VIC, 3010,  Australia, and on an unpaid leave at E\"otv\"os Lor\'and University, E-mail: \texttt{gabor.somlai@unimelb.edu.au}} \and
Binzhou Xia \thanks{School of Mathematics and Statistics, The University of Melbourne, Parkville, VIC, 3010,  Australia. E-mail: \texttt{binzhoux@unimelb.edu.au} }
\and
Sanming Zhou \thanks{School of Mathematics and Statistics, The University of Melbourne, Parkville, VIC, 3010,  Australia. E-mail: \texttt{sanming@unimelb.edu.au} }
}
\date{}

\begin{document}

\maketitle

\begin{abstract}
We study tilings of symmetric and alternating groups by conjugacy classes. For $S_n$, the case where the conjugacy class consists of transpositions was first investigated by Rothaus and Thompson in 1966 and remains open. We prove that this long-standing case is the only unresolved case for symmetric groups: every other nonidentity conjugacy class of $S_n$ fails to tile $S_n$. Our proofs combine representation-theoretic methods with combinatorial arguments concerning products of permutations. We also completely resolve the corresponding problem for alternating groups in a stronger form: for $n\geq5$, no normal subset of $A_n$ that avoids the identity can tile $A_n$.
\iffalse
We study the problem of whether a conjugacy class tiles the symmetric group or alternating group. For $S_n$, the case where the conjugacy class consists of transpositions was first investigated by Rothaus and Thompson in 1966 and remains open. We prove that this long-standing case is the only unresolved case for symmetric groups: every other nonidentity conjugacy class of $S_n$ fails to tile $S_n$. The proof combines representation-theoretic methods with combinatorial arguments concerning products of permutations. We also completely resolve the corresponding problem for alternating groups in a stronger form: for $n\geq5$, no normal subset of $A_n$ that avoids the identity can tile $A_n$.
\fi
\iffalse
    In this paper, we study when a conjugacy class or a normal subset tiles the symmetric or alternating group. The case of transpositions was first investigated by Rothaus and Thompson, and in general, it remains open. We show that every other conjugacy class of a symmetric group does not tile. We also prove that normal subsets in $A_n$ ($n\ge 5)$ not containing the identity element cannot tile.
    %Further, we prove that no nontrivial conjugacy class of a symmetric group can tile, with the possible exception of the conjugacy class of transpositions, which we investigated as our model case.

 %We initiate the study of the same question for linear groups, but we do not attempt to prove broad nonexistence results that do not follow from earlier results in representation theory.
\fi
\end{abstract}

\section{Introduction}

Let $G$ be a group. A pair $(X,Y)$ of subsets of $G$ is called a \emph{tiling} of $G$ if every element of $G$ can be written uniquely in the form $xy$, where $x\in X$ and $y\in Y$. We say that $X$ \emph{left-tiles} $G$ if $(X,Y)$ is a tiling of $G$ for some subset $Y$. We may also say that $X$ is a left tiling complement of $Y$ and $Y$ is a right tiling complement of $X$. Similarly, $Y$ \emph{right-tiles} $G$ if $(X,Y)$ is a tiling of $G$ for some subset $X$. A subset $S$ is said to \emph{tile} $G$ or $S$ is a \emph{tile} of $G$ if it either left-tiles or right-tiles $G$. If $S$ is inverse-closed or normal (that is, conjugation-closed), then $S$ left-tiles $G$ if and only if it right-tiles $G$. Moreover, if $X$ is normal, then a set $Y$ is a right tiling complement of $X$ if and only if it is a left tiling complement of $X$.

In 1966, Rothaus and Thompson \cite{rothausthompson} first investigated whether the conjugacy class of transpositions tiles $S_n$. %They also considered the closely related set obtained by adjoining the identity.
Despite subsequent work (see~\cite{transpositions}), this Rothaus--Thompson problem remains open. In fact, Fang and Xia~\cite{transpositions} conjectured that, for $n\geq4$, the conjugacy class of transpositions does not tile $S_n$.

Motivated by this question, we formulate the following problem.

\begin{problem}[Generalized Rothaus--Thompson problem]
\label{prob:generalized-RT}
For each $n$, determine the nonidentity conjugacy classes of $S_n$ that tile $S_n$.
\end{problem}

Our first main result shows that the conjugacy class of transpositions is the only unresolved case of Problem~\ref{prob:generalized-RT}.
\begin{theorem}\label{thm:symmetric}
Let $C\ne\{1\}$ be a conjugacy class of $S_n$. If $C$ tiles $S_n$, then $C=\mathcal{C}_{(2^1,1^{n-2})}$ is the set of transpositions.
\end{theorem}

The proof of Theorem~\ref{thm:symmetric} combines representation-theoretic methods with combinatorial arguments concerning products of permutations. The latter also lead to the complete solution of the analogous problem for alternating groups in the following stronger form.

\begin{theorem}\label{Thm:Alternating}
Let $S$ be a normal subset of $A_n$ with $n\geq5$ such that $1\notin S$. Then $S$ does not tile $A_n$.
\end{theorem}

For an inverse-closed subset $S$ of a group $G$ not containing the identity element, let $\Gamma=\Cay(G,S)$ be the Cayley graph of $G$ with connection set $S$, which is defined to have vertex set $G$, in which $a$ is adjacent to $b$ if and only if $b=sa$ for some $s\in S$. A subset $T$ is called a \emph{perfect code} in $\Gamma$ if $T$, together with the neighbours of its elements, covers every vertex of $\Gamma$ exactly once. It is called a \emph{total perfect code} \cite{totalperfectcodes} if the neighbours of the elements of $T$ cover every element of $G$ exactly once. More generally, for an integer $e\geq1$, a subset $T$ is called a \emph{total perfect $e$-code} in $\Gamma$ if the sets $B_e(t)\setminus\{t\}$, where $t\in T$ and $B_e(t)$ is the set of vertices at distance at most $e$ from $t$, form a partition of the vertex set of $\Gamma$. Thus, a total perfect $1$-code is precisely a total perfect code, and $T$ is a total perfect $e$-code in $\Gamma$ if and only if $(B_e(1)\setminus\{1\},T)$ is a tiling of $G$. We say that the Cayley graph $\Cay(G,S)$ is \emph{normal} if $S$ is normal in $G$. If $\Gamma$ is normal, then $B_e(1)\setminus\{1\}$ is a normal subset of $G$. Hence, Theorem~\ref{Thm:Alternating} yields the following corollary.

\begin{corollary}
For integers $e\geq1$ and $n\geq5$, there is no total perfect $e$-code in any normal Cayley graph on $A_n$.
\end{corollary}

We next place our results in the broader literature on group tilings and perfect codes. Techniques from standard Fourier analysis are mainly used in the investigation of tiles of abelian groups, see \cite{kol}. Note that even the case of finite cyclic groups, where one of the main starting tools is Sands' theorem \cite{sands}, instead of Fourier analysis, is still far from being well understood. The best results available so far can be found in \cite{labalondner}. The term ``factorization'' is also used for the same concept. For references on (total) perfect codes in Cayley graphs, we recommend that the reader consult \cite{cameron,perfectcodes}.

Several special classes of perfect codes have been studied. One of these is the class of subgroup perfect codes, which has received considerable attention \cite{subgroupperfect} in recent years. We focus mainly on tilings in which one member of the tiling pair is a normal subset. Levy and Maróti \cite{levymaroti} proved that $G$ has such a tiling $(X,Y)$, where $X$ and $Y$ are both normal if and only if $G$ is a central product of $\langle X\rangle$ and $\langle Y\rangle$ with some extra conditions on the central subgroup $\langle X\rangle\cap\langle Y\rangle$. In particular, they showed that a simple group cannot be factorised as the product of two nontrivial normal subsets; see Corollary~1.3 in \cite{levymaroti}.

More generally, a pair $(X,Y)$ is called an \emph{$r$-fold tiling} of $G$ if every element of $G$ can be written as the product of an element of $X$ and an element of $Y$ in exactly $r$ ways. When the value of $r$ is not specified, we simply call $(X,Y)$ a \emph{multiple tiling} of $G$. Green and Liebeck \cite{GreenLiebeck} studied multiple tilings of $S_n$ by a conjugacy class $C$ and the Young subgroup $S_k\times S_{n-k}$, where $n>2k$. They constructed families of such tilings and conjectured that $(C,S_k\times S_{n-k})$ is a multiple tiling precisely when the cycle type of $C$ contains exactly one cycle of each length $2^i$ for $0\leq i\leq\lfloor\log_2 k\rfloor$, with all remaining cycles having length greater than $k$ \cite[Conjecture~2.3]{GreenLiebeck}. This conjecture was subsequently proved by Fang and Li \cite[Theorem~1.1]{FangLi}.

The rest of this paper is organized as follows. Section~\ref{sec:preliminaries} introduces the notation and the main techniques used in the paper. Section~\ref{sec:thmalternating} proves Theorem~\ref{Thm:Alternating}. Section~\ref{section:S_n} proves Theorem~\ref{thm:symmetric}, where representation theory is used for the two families of conjugacy classes not covered by the elementary method. Finally, Section~\ref{subsection: limitation of techniques} discusses the limitations of the elementary approach and possible extensions.

\section{Preliminaries}\label{sec:preliminaries}
In this section, we summarize the basic notation used throughout the paper and introduce some of the fundamental techniques.

All permutations act on the right, and products of permutations are evaluated from left to right.  As usual $1$-cycles are omitted unless in some cases when they are relevant. Thus, in a product $\sigma\tau$, the permutation $\sigma$ is applied first and $\tau$ second, so
$
i^{\sigma\tau}=(i^\sigma)^\tau $. We denote by $[n]$ the set of integers from 1 to $n$.
 If $\sigma$ is  a permutation in the symmetric group $S_n$, then its support is
$$
\supp(\sigma)=\{i\in[n]\mid i^\sigma\ne i\}.
$$
A nonempty subset $A$ of $[n]$ is $g$-\emph{invariant} for $g\in S_n$ if $a^g\in A$ for every $a \in A$.
Let $g$ be a permutation in  $S_n$ and let $A \subset [n]$ be a $g$-invariant set. Then the restriction of $g$ to $A$ is denoted by $g\vert_A$.

For $x,g\in G$, we write $x^g=g^{-1}xg$, and $
x^G=\{x^g\mid g\in G\}$
denotes the conjugacy class of $x$ in $G$. For subsets $X,Y\subseteq G$, we write $XY=\{xy\mid x\in X,y\in Y\}$.

The method of the paper is largely elementary, in the last two subsections we use representation theory, which is the standard technique when studying tilings with conjugacy classes.
Hence we now collect some standard notation from representation theory and facts used below. A \emph{partition} of $n$ is a weakly decreasing sequence
$\lambda=(\lambda_1,\lambda_2,\ldots,\lambda_\ell)$
of positive integers with $\lambda_1+\cdots+\lambda_\ell=n$. In this case we write $\lambda\vdash n$.
Instead of writing repeated parts explicitly, a partition  $\lambda = (\lambda_1, \lambda_2, \dots, \lambda_k)$ is condensed by collecting  terms using exponents to represent counts (multiplicities) so we may write $\lambda$ as follows:
$$\lambda = (a_1^{m_1}, a_2^{m_2}, a_3^{m_3}, \dots ,a_k^{m_k}).$$
In principle $a_i$ is a strictly decreasing sequence but for some of the cases we ignore this condition on these numbers to be decreasing but we keep the condition that they are pairwise different.

The \emph{Young diagram} of $\lambda \vdash n$ is
\[
D(\lambda)=\{(i,j)\mid 1\le i\le\ell,1\le j\le\lambda_i\}.
\]
The \emph{transpose} of a Young diagram is obtained by reflecting it across the main diagonal, thereby interchanging rows and columns;
we denote the image of a cell $b$ by $b^{\mathsf t}$.

For a pair of partitions $\lambda,\mu\vdash n$, we say that $\mu$ \emph{dominates} $\lambda$, written $\lambda\trianglelefteq\mu$, if
\begin{equation}\label{eq:partitiondominance}
    \lambda_1+\cdots+\lambda_r\le \mu_1+\cdots+\mu_r
\end{equation}
for every $r\ge1$, where missing summands in equation \eqref{eq:partitiondominance} are taken to be zero.

For every finite group the numbers of conjugacy classes and irreducible complex representations are the same, and in the case of $S_n$, the conjugacy classes  are indexed by partitions of $n$. For $\lambda\vdash n$, let $\mathcal{C}_\lambda$ denote the conjugacy class consisting of all permutations of cycle type $\lambda$,  let $\rho^\lambda$ denote the corresponding irreducible representation of $S_n$, and let $\chi^\lambda$ denote its character.

A \emph{rim hook}, also called a \emph{border strip}, of the partition $\lambda \vdash n$ is a connected sequence of boxes in a Young diagram that lies entirely on the rim (the southeast boundary) of $\lambda$.
 Its \emph{size} is its number of boxes, and its \emph{height} is one less than the number of rows it intersects. For a pair of partitions $\lambda,\mu\vdash n$, let $BST(\lambda,\mu)$ denote the set of \emph{border-strip tableaux} of shape $\lambda$ and type $\mu$. What this means is that its elements are chains
\[
\lambda=\lambda^{(0)}\supset\lambda^{(1)}\supset\cdots\supset\lambda^{(r)}=\varnothing
\]
such that $\lambda^{(i-1)}\setminus\lambda^{(i)}$ is a rim hook of size $\mu_i$ for every $i$. For $T\in BST(\lambda,\mu)$ the height $\operatorname{ht}(T)$ of such a tableau is the sum of the heights of these rim hooks.

The Murnaghan--Nakayama rule states that if $g\in S_n$ has cycle type $\mu$, then
\[
\chi^\lambda(g)
 =\sum_{T\in BST(\lambda,\mu)}(-1)^{\operatorname{ht}(T)}.
\]
A rim hook of size $2$ is a horizontal or vertical domino whose removal leaves a Young diagram. The \emph{$2$-core} of $\lambda$ is the unique partition obtained by repeatedly removing rim hooks of size $2$ until none remains. In particular, if $n=2k$ and $g$ has cycle type $(2^k)$, then $BST(\lambda,(2^k))$ is nonempty if and only if $\lambda$ has empty $2$-core. Moreover, a simple parity argument (Lemma \ref{lem:2-core--BST}) shows that all terms in the Murnaghan--Nakayama sum have the same sign, and hence
$\chi^\lambda(g)\ne0$ if and only if
$\lambda$ has empty  $2$-core or equivalently $BST(\lambda,(2^k))$ is nonempty.
We refer to \cite{James,reprtheoryofsymmetric} for these facts.

\begin{example}
For $\lambda=(5,3)$, we may successively remove rim dominoes as follows:
\[
(5,3)\supset(3,3)\supset(3,1)\supset(1,1)\supset\varnothing.
\]
The first three removed dominoes are horizontal, and the last one is vertical. Thus the $2$-core of $(5,3)$ is empty. By combining this argument with the fact that in this case the parity of $h(T)$ is the same for every $T\in BST((5,3),(2^4))$, the Murnaghan--Nakayama rule gives $ \chi^{(5,3)}(c) \ne 0$ for every $c\in\mathcal{C}_{(2^4)}$.
\end{example}
Let $H$ be a subgroup of a finite group $G$ and $\delta$ a representation of $H$ and $\chi$ the corresponding character.  The induced  representation and the character is denoted by $\delta^G$ and $\chi^G$, respectively. The trivial character of the group $K$ is denoted by $1_K$.
In particular, if $\alpha=(\alpha_1,\ldots,\alpha_r)\vdash n$ and
$H$ is a Young subgroup of the form $ H_\alpha=S_{\alpha_1}\times\cdots\times S_{\alpha_r}\le S_n,
$
then Young's rule gives
\[
1_{H_\alpha}^{S_n}
 =\sum_{\lambda\vdash n}K_{\lambda,\alpha}\chi^\lambda,
\]
where $K_{\lambda,\alpha}$ is a Kostka number. We have $K_{\lambda,\alpha}\ne0$ if and only if $\lambda$ dominates $\alpha$.

We use the notation $\mathbb{C}[G]$ for group algebra of the group $G$ and for a finite nonempty subset $A$ of $G$, the sum of elements of $A$ will be denoted by $\underline{A}$.
We will also use the following standard group-algebra consequence of Schur's lemma. If $(C,T)$ is a tiling of a finite group $G$, then
$\underline{C}\cdot \underline{T}=\underline{G}$.
%\[
%\left(\sum_{c\in C}c\right)\left(\sum_{t\in T}t\right)=\sum_{g\in G}g.
%\]
If $C=x^G$ is a conjugacy class and $\rho$ is an irreducible representation of $G$ with character $\chi$, then
\[
\rho\left(\underline{C} \right) =\frac{|C|\chi(x)}{\chi(1)}I,
\]
where $I$ denotes the identity matrix.
Now $\underline{G}$ vanishes on every nontrivial irreducible representation. Therefore, whenever $\rho$ is a nontrivial representation of $G$ and $\chi(x)\ne0$, we have
\[
\sum_{t\in T}\rho(t)=0,
\]
which of course coincides with the Fourier techniques for Abelian groups.

We now introduce the elementary observation that will be our basic method.
\begin{lemma}\label{Lem:NecessaryCondition}
Let $G$ be a group, and let $S$ be a subset of $G$ such that $1\not\in S$ and $S\subseteq S^{-1}S$. Then $S$ does not left-tile $G$.
\end{lemma}
\begin{proof}
Suppose, for a contradiction, that $(S,Y)$ is a tiling of $G$. Since $(S,Yg)$ is also a tiling for every $g\in G$, we may assume without loss of generality that $1\in Y$. There exist $s\in S$ and $y\in Y$ such that $1=sy$, and hence $s^{-1}=y\in Y$. Therefore $s=1(s^{-1})^{-1}\in YY^{-1}$. Since $1\notin S$, $s$ is a nonidentity element of $S$.

If $s_1y_1=s_2y_2$ for $s_1,s_2\in S$ and $y_1,y_2\in Y$, uniqueness of the tiling implies $s_1=s_2$ and $y_1=y_2$. Consequently,
\[
(S^{-1}S)\cap(YY^{-1})=\{1\}.
\]
However, the nonidentity element $s$ belongs to both sets because $S\subseteq S^{-1}S$, a contradiction.
\end{proof}
For certain conjugacy classes $C$ of $A_n$,
it is known that $A_n\setminus \{ 1\} \subseteq C^2$,  see
\cite{LT2023} and the references therein. In this case  $C\subseteq C^2$, which guarantees that $C$ is not a tile if $C=C^{-1}$.
These results are special cases of Thompson's conjecture, which asserts that every finite
non-abelian simple group $G$ contains a conjugacy class $C$ such that
$C^2=G$. The conjecture first appeared as Problem~9.24 in the ninth
edition of the \emph{Kourovka Notebook} in 1984, see
\cite[Problem~9.24]{Kourovka}. A recent preprint by Liao, Wang and Zhang \cite{LWZ2026} claims to
settle the remaining cases, thereby completing the proof of the
conjecture.
Larsen and Tiep have established large families of cases of
Thompson's conjecture using character ratio estimates for finite
classical groups \cite[Theorem~7.7]{LTuniform}.
Their subsequent work proves the conjecture for all sufficiently
large finite non-abelian simple groups
\cite[Theorem~D]{LTasymptotic}.
For our purposes, we only need the following result of Larsen and
Tiep \cite[Theorem~1.1]{LT2023}.

\begin{theorem}[Larsen--Tiep]\label{Lem:LT2023}
Let $n\geq7$ be an odd integer, and let $C$ and $D$ be (possibly equal) conjugacy classes of $A_n$, each consisting of $n$-cycles. Then $CD\supseteq A_n\setminus\{1\}$.
\end{theorem}
We would like to have a complete picture so we turn to more elementary but more effective tools.
\section{Proof of Theorem \ref{Thm:Alternating}}\label{sec:thmalternating}
We say that a subset $S$ of a permutation group has the \emph{generalized square container property (GSCP)} if, for each $t\in S$, there exists $s\in S$ with $\mathrm{supp}(s)=\mathrm{supp}(t)$ such that $st\in S$.
\begin{lemma}\label{Lem:GSCP}
Let $g\in S_n$ with one of the following cycle types:
\begin{enumerate}[\rm(a)]
\item $((2k-1)^1,1^{n-2k+1})$, where $k\geq1$ is an integer,
\item $((2k)^1,2^1,1^{n-2k-2})$, where $k\geq1$ is an integer,
\item $((2k)^1,4^1,1^{n-2k-4})$, where $k\geq2$ is an integer,
\item $((2k)^1,(2\ell)^1,1^{n-2k-2\ell})$, where $k\geq3$ and $\ell\geq3$ are integers.
\end{enumerate}
Then $g^{S_n}$ has GSCP.
\end{lemma}

\begin{proof}
Since $g^{S_n}$ is closed under conjugation, it is enough to find one pair $s,t\in g^{S_n}$ such that $st\in g^{S_n}$. Indeed, conjugating this pair by an element that sends $t$ to an arbitrary element of $g^{S_n}$ produces the required pair for that element. We now give such pairs in cases~(a)--(d), writing permutations in cycle notation and omitting $1$-cycles.

(a). Let $s$ be a $(2k-1)$-cycle, and let $t=s^{-2}$. Then $s$ and $t$ have the required properties, since both $t$ and $st=s^{-1}$ are $(2k-1)$-cycles.

(b). Let $s=(1,2)(3,4,\ldots,2k+2)$ and $t=(1,4)(2,3,5,6,\ldots,2k+2)$. Then
\[
st=(1,3)(2,4,6,\ldots,2k,2k+2,5,7,9,\ldots,2k+1)
\]
also has the cycle type $((2k)^1,2^1,1^{n-2k-2})$.

(c). Let $s=(1,2,3,4)(5,6,\ldots,2k+4)$ and $t=(1,3,5,7)(2,4,6,8,9,\ldots,2k+4)$. Then
\[
st=(1,4,3,6)(2,5,8,10,12,\ldots,2k+4,7,9,11,\ldots,2k+3),
\]
as desired.

(d). For convenience, label the points in the support outside $\{1,\ldots,12\}$ as $$a_1,\ldots,a_{k-3},b_1,\ldots,b_{k-3},~~ c_1,\ldots,c_{\ell-3},d_1,\ldots,d_{\ell-3}.$$
Let
\begin{align*}
s&=(1,2,3,4,5,6,a_1,b_1,\ldots,a_{k-3},b_{k-3})(7,8,9,10,11,12,c_1,d_1,\ldots,c_{\ell-3},d_{\ell-3}),\\
t&=(1,a_{k-3},\ldots,a_1,12,9,11,8,4,b_{k-3},\ldots,b_1)(10,d_{\ell-3},\ldots,d_1,6,3,5,2,7,c_{\ell-3},\ldots,c_1).
\end{align*}
Empty indexed lists are omitted, thus the same formulas also cover the  cases $k=3$ or $\ell=3$. A direct calculation gives
\[
st=(1,7,4,2,5,3,b_{k-3},a_{k-3},\ldots,b_1,a_1)(6,12,10,8,11,9,d_{\ell-3},c_{\ell-3},\ldots,d_1,c_1),
\]
which completes the proof.
\end{proof}

Let us recall a well-known fact on the conjugacy classes of $A_n$ and $S_n$. For $t\in S_n$, the conjugacy class of $t$ in $S_n$ splits into two conjugacy classes in $A_n$ if and only if in the cycle structure of $t$, all cycles (including  $1$-cycles) have distinct odd lengths.

\begin{proof}[\bf Proof of Theorem~\ref{Thm:Alternating}]
Let $S$ be a normal subset of $A_n$ not containing $1$, and let $t$ be an arbitrary element of $S$. By Lemma~\ref{Lem:NecessaryCondition}, we only need to prove that there exists $s\in S$ with $st\in S$ since this shows $t\in S^{-1}S$. Let $t=t_1\cdots t_k$ be the cycle decomposition of $t$, where $t_1,\ldots,t_k$ are disjoint cycles (including the $1$-cycles) such that $|t_1|\leq\cdots\leq|t_{k-1}|\leq|t_k|$.

First assume that $t^{S_n}$ splits into two conjugacy classes in $A_n$. Then $t_1,\ldots,t_{k-1},t_k$ are odd cycles of pairwise distinct lengths. Consequently, each $t_i$ is conjugate to $t_i^{-2}$ in $\operatorname{Sym}(\operatorname{supp}(t_i))$, and so
$t_1^{-2}\cdots t_{k-1}^{-2}=(t_1\cdots t_{k-1})^g$
for some permutation $g$ with support contained in
$\operatorname{supp}(t_1)\cup\cdots\cup\operatorname{supp}(t_{k-1})$.
Moreover, $t_1^{-1}\cdots t_{k-1}^{-1}=(t_1\cdots t_{k-1})^h$
for some permutation $h$ with
$\supp(h)\subseteq \supp(t_1)\cup\cdots\cup\supp(t_{k-1})$.
Let $\Delta=\supp(t_k)$. If $|\Delta|\geq7$, then Theorem~\ref{Lem:LT2023} implies
\begin{equation}\label{Eqn:LT2023}
t_k\in((t_k^x)^{-1})^{\Alt(\Delta)}(t_k^y)^{\Alt(\Delta)}
\end{equation}
for all $x,y\in\Sym(\Delta)$.

If $|\Delta|=5$, then~\eqref{Eqn:LT2023} follows by multiplying representatives of the two $A_5$-classes of $5$-cycles. Each of the four possible products of these classes contains a $5$-cycle. Moreover, if the longest nontrivial cycle has length at most $3$, then the distinct odd cycle lengths are contained in ${1,3}$, forcing $n\leq4$. Hence~\eqref{Eqn:LT2023} always holds.

Take an $x\in\Sym(\Delta)$ of the same parity as $g$, and a $y\in\Sym(\Delta)$ of the same parity as $h$. By~\eqref{Eqn:LT2023}, there exist $u\in(t_k^x)^{\Alt(\Delta)}$ and $v\in(t_k^y)^{\Alt(\Delta)}$ such that $t_k=u^{-1}v$. In particular,
$\supp(u)=\supp(v)=\Delta=\supp(t_k)$.

Now set
$$
s=t_1^{-2}\cdots t_{k-1}^{-2}u
=(t_1\cdots t_{k-1})^gu
\in(t_1\cdots t_{k-1})^gt_k^{x\Alt(\Delta)}.
$$
Since $gx$ is even, $s\in(t_1\cdots t_{k-1}t_k)^{A_n}=t^{A_n}\subseteq S$. Further, since $hy$ is even, we obtain
$$
st
=t_1^{-2}\cdots t_{k-1}^{-2}u(t_1\cdots t_{k-1}t_k)
=t_1^{-1}\cdots t_{k-1}^{-1}(ut_k)
=(t_1\cdots t_{k-1})^hv
\in t^{A_n}\subseteq S.
$$
Next assume that $t^{S_n}=t^{A_n}$. In this case, $t^{A_n}$ is inverse closed. Let $t_{i_1},\ldots,t_{i_\ell}$ be the even cycles among $t_1,\ldots,t_k$. Then $\ell$ is even since $t\in A_n$. By Lemma~\ref{Lem:GSCP}, $t_i^{S_n}$ and $(t_{i_{2j-1}}t_{i_{2j}})^{S_n}$ have GSCP for each
$i\in\{1,\ldots,k\}\setminus\{i_1,\ldots,i_\ell\}$
and $j\in{1,\ldots,\ell/2}$.

Since $t^{S_n}$ does not split in $A_n$, either $t$ contains an even cycle or two of its odd cycles have the same length. In the former case, the number of even cycles is even, so there are at least two of them. Thus, after relabelling the cycles if necessary, we may assume that
$\bar t:=t_{k-1}t_k$
is either the product of two even cycles or two odd cycles of the same length. In either case,
$
\bar t^{\operatorname{Sym}(\operatorname{supp}(\bar t))}
=\bar t^{\operatorname{Alt}(\operatorname{supp}(\bar t))}$.

Now we choose the corresponding GSCP witnesses independently on the disjoint supports. For every odd cycle $t_i$ among $t_1,\ldots,t_{k-2}$, we choose $x_i\in\operatorname{Sym}(\operatorname{supp}(t_i))$ such that, for $s_i:=t_i^{x_i}$, the product $s_it_i$ is a cycle whose length is $|\operatorname{supp}(t_i)|$. We pair the even cycles among $t_1,\ldots,t_{k-2}$, and for every such pair $t_i,t_j$, we choose $x_{ij}\in\operatorname{Sym}(\operatorname{supp}(t_i)\cup\operatorname{supp}(t_j))$ such that, for $s_{ij}:=(t_it_j)^{x_{ij}}$, the product $s_{ij}t_it_j$ has the same cycle structure as $t_it_j$. Depending on whether the product of these permutations is in $A_n$ or $S_n\setminus A_n$, we may choose $y$ of the appropriate parity such that $\bar t^y$ and $\bar t^y\bar t$ have the same cycle structure as $\bar t$. This implies the existence of $s\in S$ with $st\in S$.
\end{proof}
\section{Proof of Theorem \ref{thm:symmetric}}\label{section:S_n}
The proof of Theorem~\ref{thm:symmetric} proceeds as follows. First, Theorem~\ref{Thm:Alternating} reduces the problem to conjugacy classes consisting of odd permutations. Lemma~\ref{lem:GSCP for symmetric} then shows that, apart from the class of transpositions, it is enough to prove the inclusion $S(1,2)\subseteq S^2$. Lemma~\ref{lem:cycle structure} describes the possible cycle structures occurring in $S(1,2)$ and reduces this inclusion to four cases. Subsection~\ref{subsection:case1a} treats the case of cutting an even cycle, while Subsection~\ref{subsection:case1b} deals with gluing two cycles whose total length is odd. Subsection~\ref{subsection:case2a} handles the gluing two cycles whose total length is even, and Subsection~\ref{subsection:case2b} treats the remaining case of cutting an odd cycle. These arguments exclude all relevant conjugacy classes, except certain classes of involutions. In Subsection~\ref{subsection:involutions with fixed points}, explicit local factorizations are used to exclude involutions having at least two fixed points. Finally, the two remaining cases are settled by representation-theoretic arguments. In Subsection~\ref{subsection:fixed-point-free-involutions}, Theorem~\ref{thm:(2^k)} excludes fixed-point-free involutions using the Murnaghan--Nakayama rule, Lemma~\ref{lem:2-core--BST}, and Young's rule, while in Subsection~\ref{subsection:one-fixed-point}, Theorem~\ref{thm:(2^k,1^1)} excludes involutions with one fixed point by combining a group-algebra argument with Lemma~\ref{lem:ket konjugalt osztaly negyezete} and a version of Lemma \ref{Lem:NecessaryCondition}.

We first dispose of the cases $n\leq4$. For $n\leq2$, the assertion is immediate. If $n=3$, the only nonidentity conjugacy class other than the class of transpositions is the class of $3$-cycles, and it satisfies $C\subseteq C^2=C^{-1}C$. If $n=4$, the classes of $3$-cycles and double transpositions also satisfy $C\subseteq C^2=C^{-1}C$. Hence none of these classes tiles, by Lemma~\ref{Lem:NecessaryCondition}.

It remains to consider the class $C$ of $4$-cycles in $S_4$. We have $C=C^{-1}$ and $C^2=A_4$. Indeed, $(1,2,3,4)(1,4,3,2)=1$, $(1,2,3,4)^2=(1,3)(2,4)$, and $(1,2,3,4)(1,2,4,3)$ is a $3$-cycle. Thus $C^2$ contains the identity, a double transposition, and a $3$-cycle; since $C^2$ is normal in $S_4$, it contains every element of $A_4$. Suppose that $(C,T)$ is a tiling of $S_4$, normalized so that $1\in T$. The uniqueness of the tiling gives
\[
(C^{-1}C)\cap(TT^{-1})=\{1\}.
\]
Since $C^{-1}C=A_4$, the set $T$ contains at most one element from each coset of $A_4$ in $S_4$, and hence $|T|\leq2$. This contradicts $|T|=|S_4|/|C|=4$. Thus Theorem~\ref{thm:symmetric} holds for $n\leq4$, and from now on we assume that $n\geq5$.

If $S$ is a conjugacy class of $S_n$ contained in $A_n$, then $S$ is normal in $A_n$. By Theorem~\ref{Thm:Alternating}, it does not tile $A_n$. If $(S,Y)$ tiled $S_n$, then restricting the factorization to the even elements would give the tiling $(S,Y\cap A_n)$ of $A_n$. Hence $S$ cannot tile $S_n$.
Notice that this holds for every subset of the group that is not a generating set. It tiles the whole group if and only if it tiles the subgroup it generates.
We may therefore restrict our attention to conjugacy classes of $S_n$ contained in $S_n\setminus A_n$.

Let $S$ be such a conjugacy class. Notice that $S$ is inverse-closed, that is, $S=S^{-1}$.
Now we generalize Lemma \ref{Lem:NecessaryCondition}.
\begin{lemma}\label{lem:GSCP for symmetric}
Let $S \subset S_n\setminus A_n$ be a conjugacy class of $S_n$ which is not the conjugacy class of transpositions. Assume $S~(1,2)\subset S^{-1}S=S^2$. Then $S$ is not a tile of $S_n$.
\end{lemma}
\begin{proof}
Suppose, for a contradiction, that $(S,Y)$ is a tiling of $S_n$, and normalized so that $1\in Y$. Write $(1,2)=sy$ with $s\in S$ and $y\in Y$. Since $S=S^{-1}$, we have $y=s^{-1}(1,2)\in S(1,2)$. Moreover, $y\ne1$, because otherwise $(1,2)=s\in S$, contrary to the assumption that $S$ is not the conjugacy class of transpositions. Since $1,y\in Y$, the element $y$ lies in $YY^{-1}$. Hence
\[
1\ne y\in (YY^{-1})\cap S(1,2)\subseteq (YY^{-1})\cap(S^{-1}S),
\]
contradicting the uniqueness of the tiling.
\end{proof}
Of course, $(1,2)$ in Lemma \ref{lem:GSCP for symmetric} can be replaced by any permutation of $S_n\setminus (A_n \cup S)$.

From now on, let $S$ be a conjugacy class of $S_n$ not contained in $A_n$. We plan to prove that $S(1,2)$ is contained in $S^2$. Since $S^2$ is a normal set in $S_n$ we only have to show that the conjugacy classes that intersect $S(1,2)$ are contained in $S^2$. A conjugacy class of $S_n$ is determined by the cycle structure of its elements, so we need to understand the cycle structure of the elements of $S(1,2)$.

For the next lemma, we regard the cycle structure of $\pi\in S_n$ as a $k$-tuple of positive integers $(\lambda_1,\lambda_2, \ldots ,\lambda_k)$ for some $1\le k \le n$ such that $\sum_{i=1}^k \lambda_i=n$. In this case, we do not impose any further restriction on these $k$-tuples, that is, the entries need not be written in increasing or decreasing order.
\begin{lemma}\label{lem:cycle structure}
    Let $\pi$ be a permutation in $S_n$ of cycle structure $(\lambda_1,\lambda_2, \ldots ,\lambda_k)$.
\begin{enumerate}
    \item[\rm (a)] If $1$ and $2$ are contained in the cycle of length $\lambda_1$, then the cycle structure of $\pi(1,2)$
is $(\eta_1,\eta_2,\lambda_2, \ldots ,\lambda_k)$ with $\lambda_1=\eta_1+\eta_2$.
Moreover $\eta_1-1$ is the number of elements in the cycle between $1$ and $2$.
\item[\rm (b)]
If $1$ is contained in a cycle of length $\lambda_1$ and $2$ is contained in a cycle of length $\lambda_2$, then the cycle structure of $\pi(1,2)$
is $(\lambda_1+\lambda_2, \lambda_3, \ldots ,\lambda_k)$.
\end{enumerate}
\end{lemma}
\begin{proof}
   It is straightforward to verify both statements.
\end{proof}
Let $\pi$ be an element of $S$. In particular $\pi$ is an odd permutation.
We determine whether $\pi(1,2)$ belongs to $S^2$ considering the following four exhaustive cases.
\begin{enumerate}
    \item \textbf{Case 1a:} The points $1$ and $2$ lie in the same cycle, and its length $\lambda_1$ is even.
    \item \textbf{Case 1b:} The points $1$ and $2$ lie in distinct cycles of lengths $\lambda_1$ and $\lambda_2$, and $\lambda_1+\lambda_2$ is odd.
    \item \textbf{Case 2a:} The points $1$ and $2$ lie in distinct cycles of lengths $\lambda_1$ and $\lambda_2$, and $\lambda_1+\lambda_2$ is even.
    \item \textbf{Case 2b:} The points $1$ and $2$ lie in the same cycle, and its length $\lambda_1$ is odd.
\end{enumerate}

\subsection{Case 1a}\label{subsection:case1a}

The following statement is attributed to Gleason, see Proposition 4 in \cite{Gleason}. Besides its role here, this statement also shows the well-known fact that every element of $A_n$ is a commutator in $S_n$.
\begin{lemma}\label{lem:gleason}
Let $\sigma$ be a permutation in $A_n$. Then $\sigma$ can be written as the product of two $n$-cycles.
\end{lemma}
As a corollary of this lemma, we obtain the following statement.
\begin{prop}
Let $\sigma$ be an even permutation of cycle structure $(\eta_1,\eta_2,\lambda_2,\ldots,\lambda_k)$ with $2\mid \eta_1+\eta_2=\lambda_1$. If $S$ is the conjugacy class of permutations of cycle structure $(\lambda_1,\lambda_2,\ldots,\lambda_k)$, then $\sigma\in S^2$.
\end{prop}
\begin{proof}
Write $\sigma=\sigma_0\sigma_1$, where $\sigma_0$ has cycle structure $(\eta_1,\eta_2)$ and $\sigma_1$ has cycle structure $(\lambda_2,\ldots,\lambda_k)$, with disjoint supports. Since $\lambda_1$ is even, we have $\eta_1\equiv\eta_2\pmod 2$, and hence $\sigma_0$ is in $A_n$. Lemma~\ref{lem:gleason} therefore expresses $\sigma_0$ as a product of two $\lambda_1$-cycles supported on $\supp(\sigma_0)$.

The permutation $\sigma_1=\sigma\sigma_0^{-1}$ is also even. The proof of Theorem~\ref{Thm:Alternating} shows that an even permutation $\rho$ can be written as a product of two permutations having the same cycle structure and support as $\rho$. Indeed, choose $a$ with $a\rho$ in the same conjugacy class and write $\rho=a^{-1}(a\rho)$.

Notice that we may only directly use the argument for cycles of odd length at least 5 or product of disjoints cycles with support at least 5. The cases on supports of size at most four are immediate.

Applying this to $\sigma_1$ and combining the two factorizations on disjoint supports gives the required factorization of $\sigma$ into two elements of $S$.
\end{proof}
\subsection{Case 1b}\label{subsection:case1b}
Let us assume that $S$ consists of the permutations of cycle structure $(\lambda_1,\lambda_2,\ldots , \lambda_k)$, where $2 \nmid \lambda_1+\lambda_2$.
\begin{lemma}
For all nonnegative integers $k$ and $\ell$,
    \begin{equation*}
    \begin{split}
            &(x_1,\dots,x_{2k},1)(2,y_1,\dots,y_{2\ell},3)(x_1,\dots,x_{2k},2)(1,y_1,\dots,y_{2\ell},3)
            \\ &= (x_1,x_3,\dots,x_{2k-1},2,y_2,y_4,\dots,y_{2\ell},1,x_2,x_4,\dots,x_{2k},y_1,y_3,\dots,y_{2\ell-1},3).
\end{split}
   \end{equation*}
\end{lemma}
\begin{proof}
This follows by tracing the image of each point.
\end{proof}
It follows that $S^2$ contains the conjugacy class of cycle structure $(\lambda_1+\lambda_2,\lambda_3,\ldots,\lambda_k)$. Indeed, the remaining cycles form an even permutation and can be factored independently on their disjoint support, as in the proof of Case~1a.
\subsection{Case 2a}\label{subsection:case2a}
In this case, we seek a pair of permutations of cycle structure $(\lambda_1,\lambda_2, \ldots, \lambda_k)$ whose product is of cycle structure $(\lambda_1+\lambda_2,\lambda_3, \ldots, \lambda_k)$ (thus gluing two cycles). These cycles are both odd or even, i.e. $\lambda_1\equiv \lambda_2 \pmod{2}$.

Let $\sigma$ be an element of $S$.
Let us write $\sigma=\sigma_0'\sigma_1'$, where the cycle structure of $\sigma_0'$ is $(\lambda_1, \lambda_2)$ and the cycle structure of
$\sigma_1'$ is $(\lambda_3, \ldots ,\lambda_k)$ and $\supp(\sigma_0')\cap \supp(\sigma_1')=\emptyset$.

Since $\lambda_1+\lambda_2$ is even, $\sigma_0'$ is an even permutation. We cannot expect to find two permutations of the same cycle structure whose product is a single cycle of length $\lambda_1+\lambda_2$ so we modify the decomposition of $\sigma$.

Thus $\sigma_1'$ is odd, so if we write it as the product of disjoint cycles, then one of them is a cycle of even length. Let us then write $\sigma_1'=\pi \sigma_1$, where $\supp(\pi)\cap \supp(\sigma_1)=\emptyset$ and $\pi$ is a cycle of length $2s=\lambda_3$. We can then write $\sigma=\sigma_0 \sigma_1$, where $\sigma_0=\sigma_0' \pi$. Note that $\sigma_1$ is an even permutation.

\textbf{Subcase 1:} We first consider the case $\lambda_1\equiv\lambda_2\equiv1\pmod 2$.
\begin{lemma}
For every pair of odd integers $u,v\ge3$ and every even integer $w\ge4$, there exist permutations $\alpha_1$ and $\alpha_2$ of cycle structure $(u,v,w)$, with $\supp(\alpha_1)=\supp(\alpha_2)$, such that $\alpha_1\alpha_2$ has cycle structure $(u+v,w)$.
\end{lemma}
\begin{proof}
For all nonnegative integers $a,b,c$, we have
    \begin{equation*}
        \begin{split}
        &\underbrace{(1,x_1,\dots,x_{2a},2,3)(4,y_1,\dots,y_{2b},5,6)(7,8,z_1,\dots,z_{2c},9,10)}_{\alpha_1} \cdot
        \\& \underbrace{(y_1,\dots,y_{2b},1,4,5)(2,x_1,\dots,x_{2a},7,9)(6,z_1,\dots,z_{2c},8,10,3)}_{\alpha_2}
        \\= &(1,x_2,x_4,\dots,x_{2a},x_1,x_3,\dots,x_{2a-1},7,10,9,3,4,y_2,y_4,\dots,y_{2b},\\ & y_1,y_3,\dots,y_{2b-1})
        (2,6,5,z_1,z_3,\dots,z_{2c-1},8,z_2,z_4,\dots,z_{2c})=\alpha_1\alpha_2.
        \end{split}
    \end{equation*}
The factors have cycle lengths $2a+3$, $2b+3$, and $2c+4$, while their product has cycle structure $(2a+2b+6,2c+4)$.
\end{proof}

\begin{lemma}
For every pair of odd integers $u,v\ge3$, there exist permutations $\alpha_1$ and $\alpha_2$ of cycle structure $(u,v,2)$, with $\supp(\alpha_1)=\supp(\alpha_2)$, such that $\alpha_1\alpha_2$ has cycle structure $(u+v,2)$.
\end{lemma}
\begin{proof}
    \begin{equation*}
    \begin{split}
    &(1,x_1,\dots,x_{2k},2,3)(4,5,6,y_1,\dots,y_{2\ell})(7,8) \cdot (1,y_1,\dots,y_{2\ell},4,3)(2,x_1,\dots,x_{2k},6,8)\\ &(5,7) =  (1,x_2,x_4,\dots,x_{2k},x_1,x_3,\dots,x_{2k-1},6,y_2,y_4,\dots,y_{2\ell},3,y_1,y_3,\dots,\\ & y_{2\ell-1},4,7,2)(5,8).
    \end{split} \end{equation*}
\end{proof}
It remains to glue a $1$-cycle to a nontrivial odd cycle. The following identity gives the required construction:
\begin{equation*}
\begin{split}
&(1)(2,x_1,x_2,\ldots,x_{2k},3,4)(5,6,y_1,y_2,\ldots,y_{2\ell}) \cdot (2)(1,4,6,x_1,x_2,\ldots,x_{2k}) \cdot \\&(3,5,y_1,y_2,\ldots,y_{2\ell})  = (1,4,2,x_2,x_4,\ldots,x_{2k},5,x_1,x_3,\ldots,x_{2k-1}) \\ &(3,6,y_2,y_4,\ldots,y_{2\ell},y_1,y_3,\ldots,y_{2\ell-1}).
\end{split}
\end{equation*}
The left-hand side is a product of two permutations of cycle structure $(1,2k+3,2\ell+2)$, and the right-hand side has cycle structure $(2k+4,2\ell+2)$.
These constructions yield the following statement.
\begin{prop}
\begin{enumerate}[\rm(a)]
    \item
    Let $S$ be the conjugacy class of odd permutations of cycle structure $(\lambda_1,\dots, \lambda_k)$, where $\lambda_1$ and $\lambda_2$ are odd integers. Then every permutation of cycle structure $(\lambda_1+\lambda_2,\lambda_3,\dots, \lambda_k)$ is contained in $S^2$.
   % , provided that at least one of $\lambda_i$'s is an even number larger than $2$.
\item
 Let $S$ be the conjugacy class of odd permutations of cycle structure $(\lambda_1,\dots, \lambda_k)$, where $\lambda_1 \ge 2$ and $\lambda_2 \ge 4$ are even integers. Then every permutation of cycle structure $(\lambda_1+\lambda_2,\lambda_3,\dots, \lambda_k)$ is contained in $S^2$.
   \end{enumerate}

\end{prop}
\begin{proof}
The constructions above prove the first assertion except when $\lambda_1=\lambda_2=1$, which is handled immediately below. The even--even assertion follows from the subsequent constructions. In every case, the remaining even permutation is factored independently on the complementary support.
\end{proof}

Let us handle the case when $\lambda_1=\lambda_2=1$. More precisely, let $\pi$ be an odd permutation which fixes $y_1$ and $y_2$. We would like to find two permutations $\alpha_i$ ($i=1,2$) which are of the same cycle structure as $\pi$ and such that $\pi(y_1,y_2)=\alpha_1\alpha_2$.

As before we may assume the cycle structure of $\pi$ contains a cycle of length $2k$.
\begin{itemize}
    \item Let us first ``glue'' two $1$-cycles (fixed points) i.e., $k=1$: $(1)(2)(3,4)\cdot(3)(4)(1,2)=(1,2)(3,4)$
    \item Let $k \ge 2$. The following identity shows that the product of two permutations of the same support of cycle structure $(1^2,2\ell+4)$, where $\ell\ge 0$ can be of cycle structure $(2^1,2\ell+4)$.
    \begin{equation}
    \begin{split}
&(1)(2)(3,x_1,\dots,x_{2\ell},4,5,6)\cdot(4)(5)(1,x_1,\dots,x_{2\ell},3,2,6)\\ = &(2,6)(1,x_1,x_3,\dots,x_{2\ell-1},3,x_2,x_4,\dots,x_{2\ell},4,5).  \end{split}
\end{equation}
\end{itemize}

\textbf{Subcase 2:}
Let us assume now that $\lambda_1 \equiv \lambda_2 \equiv 0 \pmod{2}$.

Similarly to the previous case, we would like to show that for many triples $(a,b,c)$ of positive integers, there are permutations $\alpha_1$ and $\alpha_2$ with $\supp(\alpha_1)=\supp(\alpha_2)$ of cycle structure $(2a,2b,2c)$ such that the cycle structure of $\alpha_1\alpha_2$ is $(2a+2b,2c)$. We will be able to find such a pair of permutations when $a\ge 2$ and $b,c\ge 1$.

The following identity holds for any nonnegative integers $k,\ell,m$.
\begin{equation}
\begin{split}
   % &\underbrace{(1,2,3,y_1,y_2,\dots,y_{2l},4)(7,8,z_1,z_2,\dots,z_{2m})(5,x_1,x_2,\dots,x_{2k},6)\cdot}_{\alpha_1} \\
    %&\underbrace{(2,5,z_1,z_2,\dots,z_{2m},8,4,y_1,y_2,\dots,y_{2l})(3,7)(1,6,x_1,x_2,\dots,x_{2k})}_{\alpha_2}
    %\\= &(1,5,x_2,x_4,\dots,x_{2k},x_1,x_3,\dots,x_{2k-1})\\&(6,z_1,z_3,\dots,z_{2m-1},8,z_2,z_4,\dots,z_{2m},3,y_2,y_4,\dots,y_{2l},%y_1,y_3,\dots,y_{2l-1},2,7,4).
   &(1,2,3,y_1,y_2,\dots,y_{2\ell},4)(7,8,z_1,z_2,\dots,z_{2m})(5,x_1,x_2,\dots,x_{2k},6)\cdot \\
   &(2,5,8,4,y_1,y_2,\dots,y_{2\ell})(3,7,z_1,z_2,\dots,z_{2m})(1,6,x_1,x_2,\dots,x_{2k}) \\
   =&(2,7,4,6,8,z_2,z_4,\dots,z_{2m},z_1,z_3,\dots,z_{2m-1},3,y_2,y_4,\dots,y_{2\ell},y_1,y_3,\dots,\\ & y_{2\ell-1})
     (1,5,x_2,x_4,\dots,x_{2k},x_1,x_3,\dots,x_{2k-1}).
     \end{split}
\end{equation}
The cycle structure of both $\alpha_1$ and $\alpha_2$ is $(2\ell+4,2m+2,2k+2)$, while the cycle structure of $\alpha_1\alpha_2$ is $(2\ell+2m+6,2k+2)$, as required.
We therefore obtain the following.
\begin{lemma}
For every pair of even integers $u\ge2$ and $v\ge4$, and every even integer $w\ge2$, there exist $\alpha_1$ and $\alpha_2$ of cycle structure $(u,v,w)$ such that the cycle structure of $\alpha_1\alpha_2$ is  $(u+v,w)$.
\end{lemma}
Let us turn to the case when $u=v=2$. In this case we need to have another even cycle of length $w$ in the cycle structure. We first assume $w\ge 4$. The following equality shows that we may obtain a permutation of cycle structure $(4^1,(4+2\ell)^1)$ as the product of two permutations of the same support of cycle structure $(2^2,(4+2\ell)^1)$.
\begin{equation}
\begin{split}
&(1,2)(3,4)(5,6,x_1,x_2,\dots,x_{2\ell},7,8) \cdot (2,3)(4,5)(1,6,7,x_1,x_2,\dots,x_{2\ell},8) \\ = &(1,3,5,7)(2,6,x_2,x_4,\dots,x_{2\ell},x_1,x_3,\dots,x_{2\ell-1},8,4).
     \end{split}
\end{equation}

We now address the case where we glue two $2$-cycles of an odd permutation but the remaining even cycles are also of length 2.
Assume now that the cycle structure under consideration also contains an odd cycle.

The following equation holds for every integer $\ell\ge 0$.
\begin{equation*}
\begin{split}
&(1,2)(3,4)(5,6)(7,8,x_1,x_2,\dots,x_{2\ell},9) \cdot (1,4,9,x_1,x_2,\dots,x_{2\ell})(2,6)(3,7)(5,8) \\ = &(1,6,8,x_2,x_4,\dots,x_{2\ell},x_1,x_3,\dots,x_{2\ell-1})(2,4,7,5)(3,9).
\end{split}
\end{equation*}
Combining the preceding constructions gives the following.
\begin{prop}
Let $S$ be the conjugacy class of odd permutations of cycle structure $(\lambda_1,\dots,\lambda_k)$, where $\lambda_1\equiv\lambda_2\pmod 2$. Assume that this cycle structure contains a cycle of length at least $3$. Then every permutation of cycle structure $(\lambda_1+\lambda_2,\lambda_3,\dots,\lambda_k)$ belongs to $S^2$.
\end{prop}
\begin{proof}
If $\lambda_1$ and $\lambda_2$ are odd, use the odd--odd constructions above. If they are even and are not both equal to $2$, use the general even construction. If both equal $2$, use either the construction with an auxiliary even cycle of length at least $4$, or, when all other even cycles have length $2$, the final construction with an auxiliary odd cycle. These cases exhaust the possibilities, and the remaining cycles are handled independently on their disjoint support.
\end{proof}
\subsection{Case 2b}\label{subsection:case2b}
Let us assume that $1$ and $2$ are contained in the same cycle of length  $\lambda_1$, where $\lambda_1$ is odd. We would like to ``cut'' a cycle of length $\lambda_1$ into the product of two cycles.

We use a strategy similar to that of the previous case. Let us start with the simplest case.
Let $\sigma$ be an odd permutation of cycle structure $(\lambda_1,\lambda_2, \ldots , \lambda_k)$. Let us assume for a moment that $\lambda_1=3$. Since $\sigma\in S_n\setminus A_n$, there is $2 \le i \le k$ such that $2 \mid \lambda_i$. For simplicity, let us first assume $\lambda_2=2$. Our purpose is to find a permutation $\eta$ that has the same cycle structure as $\sigma$ such that $\sigma \eta $ is of cycle structure $(1,2, 2,\lambda_3, \ldots, \lambda_k)$.

 Consider the following product.
\begin{equation}\label{eq:3,2->1,2,2}
\begin{split}
&(1,2,5)(3,4)\cdot (1,3,2)(4,5)=\left[(1,2,3,4,5)(3,5)\right]\cdot \left[(3,5)(2,1,3,4,5)\right]\\&=(1,2,3,4,5)(2,1,3,4,5)=(1)(2,4)(3,5).
\end{split}
\end{equation}
As we can see from this equation, we can write any permutation of cycle structure $(2,2,1)$  as the product of two permutations of cycle structure $(3,2)$. Equation \eqref{eq:3,2->1,2,2} also shows that this is equivalent to observing that the product of two cycles of length $5$ with exactly one element between 3 and 5 is a permutation of cycle structure $(2,2,1)$.
We generalize this observation by constructing cycles of length $2k+2\ell+2m+5$ whose product is a permutation of cycle structure $(2k+1,2m+2,2\ell+2)$ and the number of elements between 3 and 5 in both of these cycles is $2\ell+1$.

The following two cycles have length $2(k+\ell+m)+5$.
\begin{equation*}
    \begin{split}
\alpha = (1,x_1,\dots,x_{2k},2,z_1,\dots,z_{2m},3,y_1,\dots,y_{2\ell},4,5), \\
\beta = (2,x_1,\dots,x_{2k},1,3,4,y_1,\dots,y_{2\ell},5,z_1,\dots,z_{2m}).
\end{split}
\end{equation*}
A direct calculation gives
\begin{equation*}
    \begin{split}
\alpha\beta = &(1,x_2,x_4,\dots,x_{2k},x_1,x_3,\dots,x_{2k-1})(2,z_2,z_4,\dots,z_{2m},4,z_1,z_3,\dots,\\  &z_{2m-1})  (3,y_2,y_4,\dots,y_{2\ell},y_1,y_3,\dots,y_{2\ell-1},5).
\end{split}
\end{equation*}
As a corollary we obtain the following lemma.
\begin{lemma}
Let $S\subset S_n\setminus A_n$ be the conjugacy class of permutations of cycle structure $(2a+1,\lambda_2,\ldots,\lambda_s)$, where $a\ge1$. Then $S^2$ contains every permutation of cycle structure $(a_1,2a+1-a_1,\lambda_2,\ldots,\lambda_s)$.
\end{lemma}
\begin{proof}
One of $a_1$ and $2a+1-a_1$ is odd and the other is even. Write them as $2k+1$ and $2m+2$. Since the elements of $S$ are odd, one of the remaining cycles has even length, say $2\ell+2$. In the construction above, let $\tau=(3,5)$. Multiplication by $\tau$ cuts each of the long cycles $\alpha$ and $\beta$ into cycles of lengths $2k+2m+3=2a+1$ and $2\ell+2$. Thus $\alpha\tau$ and $\tau\beta$ have the original cycle structure, while $(\alpha\tau)(\tau\beta)=\alpha\beta$ has the required split cycle structure. The remaining cycles are factored independently on the complementary support.
\end{proof}
\subsection{Involutions having fixed points}\label{subsection:involutions with fixed points}
The preceding cases exclude every odd conjugacy class except certain classes of involutions. We now consider involutions with fixed points.
Let $S$ be the conjugacy class of triple transpositions, that is, the cycle structure of its elements is $(2^3,1^{n-6})$. Let us assume that $n \ge 8$.
\begin{itemize}
    \item $(1,2)(3,4)(5,6)\cdot (1,2)=(1,2)(3,5)(4,6)\cdot (1,2)(3,6)(4,5)$,
    \item $(1,2)(3,4)(5,6)\cdot (1,3)=(1,2,3,4)(5,6)=(1,4)(2,3)(7,8) \cdot (2,4)(5,6)(7,8)$,
    \item
    $(1,2)(3,4)(5,6)\cdot (7,8)=(1,3)(2,4)(5,6) \cdot (1,4)(2,3)(7,8)$,
\item
$(1,2)(3,4)(5,6)\cdot (6,7)=(1,2)(3,4)(5,7,6) =(1,3)(2,4)(5,6)\cdot (1,4)(2,3)(6,7)$.
\end{itemize}
These four calculations correspond respectively to the cases where the transposition $(1,2)$ is an edge of the matching, joins two matching edges, joins two fixed points, or joins a matched point to a fixed point. Thus $S(1,2)\subseteq S^2$.
\begin{prop}
Let $S$ be the conjugacy class of $S_n$ of cycle structure $(2^{2k+1},1^{n-4k-2})$. Assume that $n-4k-2\ge 2$. Then $S (1,2) \subset S^2$.
\end{prop}
\begin{proof}
Apply one of the four local factorizations above and append the same $2k-2$ disjoint transpositions to both factors on the unused points. These added transpositions cancel in the product and give both factors the required cycle structure.
\end{proof}
As a corollary, we obtain the following result concerning conjugacy classes of $S_n$. In order to finish the proof of Theorem \ref{thm:symmetric} we will use representation theory to exclude the case when $n=2k$ and the conjugacy class is $\mathcal{C}_{(2^k)}$.
\begin{theorem}\label{thm:partial S_n}
Let $S \ne \{1\}$ be a conjugacy class of $S_n$ whose permutations are not of cycle structure $(2^{\lfloor n/2\rfloor},1^{\lceil n/2\rceil-\lfloor n/2\rfloor})$ or $(2^1,1^{n-2})$. Then $S$ is not a tile in $S_n$.
\end{theorem}
\subsection{Excluding fixed-point-free involutions}\label{subsection:fixed-point-free-involutions}
In this subsection, we return to the representation-theoretic method used by Rothaus and Thompson \cite{rothausthompson} and also by Fang and Xia \cite{transpositions} to handle one of the remaining special cases. With Theorem \ref{thm:(2^k)} we solve one of the last cases in the proof of Theorem \ref{thm:symmetric}. The proof of our main result will be completed in the next subsection when we prove Theorem \ref{thm:(2^k,1^1)} as well.

Let $n=2k$ and let
$C=\mathcal{C}_{(2^k)}$.
We prove that $C$ does not tile $S_n$. By Theorem~\ref{Thm:Alternating}, it would be enough to consider when $k$ is odd. However, the argument below works for all $k\ge2$.
\begin{theorem}\label{thm:(2^k)}
Let $n=2k\ge4$ and let $C=\mathcal{C}_{(2^k)}$. Then $C$ is not a tile in $S_n$.
\end{theorem}
Notice that the case $k=3$ also follows from the existence of an outer automorphism of $S_6$, which interchanges the conjugacy classes of transpositions and triple transpositions. Since the image of a tile under a group automorphism is again a tile, non-tiling of the transposition conjugacy class in $S_6$, a case which is known \cite{transpositions}, implies non-tiling of $\mathcal{C}_{(2^3)}$.

\begin{proof}
Let $c\in C$ and  $\lambda\vdash n$, let $\chi^\lambda$ denote the corresponding irreducible character of $S_n$. We first determine when $\chi^\lambda(c)$ is nonzero. The Murnaghan--Nakayama rule \cite{reprtheoryofsymmetric} gives
\begin{equation}\label{eq:murnaghan-nakayama}
\chi^\lambda(c)=\sum_{T\in BST(\lambda,(2^k))}(-1)^{\operatorname{ht}(T)},
\end{equation}
where $BST(\lambda,(2^k))$ is the set of border-strip tableaux of shape $\lambda$ and type $(2^k)$, see \cite{James}. %Border-strip tableaux are basically a sequence of nested Young diagrams where each step adds a valid border strip, which is a sequence of connected cells on the boundary of the tableau, see \cite{James}.

Every border strip occurring here is a domino of size $1\times 2$ or $2\times 1$. A horizontal domino contributes zero to the height, whereas a vertical domino contributes one. The parity of the number of vertical dominoes is independent of the chosen domino tableaux. In each row, the parity of the number of boxes belonging to vertical dominoes, is determined by the row length. More precisely the parity of the vertical dominoes whose top cell intersects the first row depends on the parity of the length of the first row and then the same quantity can recursively be calculated for the remaining rows.
Hence we may sum over the rows to determine the parity of the total number of vertical dominoes. Consequently, all terms in~\eqref{eq:murnaghan-nakayama} have the same sign.

The same argument applies to cycle type $(2^k,1)$. Fix the order $2,\ldots,2,1$ of the border-strip lengths. The final border strip of size one is necessarily the unique box of the partition $(1)$. As above, the parity of the number of vertical dominoes is independent of the chosen domino tableau, and hence all terms in the Murnaghan--Nakayama formula have the same sign. Consequently,
\[
\chi^\lambda(c)\ne 0
\quad\Longleftrightarrow\quad
BST(\lambda,(2^k,1))\ne\emptyset.
\]
Equivalently, this holds precisely when the \(2\)-core of \(\lambda\) is \((1)\).

We obtain the following  lemma.
\begin{lemma}\label{lem:2-core--BST}
Let $c$ be a permutation in $\mathcal{C}_{(2^k)}$ or
$\mathcal{C}_{(2^k,1)}$. Then
\[
\chi^\lambda(c) \ne 0
\]
if and only if
\[
BST(\lambda,(2^k)) \ne \emptyset
\quad\text{or}\quad
BST(\lambda,(2^k,1)) \ne \emptyset,
\]
respectively. Equivalently, $\chi^\lambda(c)\ne 0$ precisely when
the $2$-core of $\lambda$ is empty or $(1)$, respectively.
\end{lemma}
%Using previously defined objects, we can just say that the $2$-core of a Young diagram is the unique residual partition obtained by repeatedly removing 2-box rim hooks -which are just dominos such that removing them still result in a Young diagram - on the boundary until no more can be removed, see \cite{James}.

Suppose, for a contradiction, that $(C,T)$ is a tiling of $S_n$. In the group algebra $\mathbb{C}[S_n]$, this means
\begin{equation}\label{eq:group algebra}
\underline{C} \cdot \underline{T} =\underline{S_n}
%\left(\sum_{c\in C}c\right)\left(\sum_{t\in T}t\right)=\sum_{g\in %S_n}g.
\end{equation}
Let
$\rho^\lambda$ be a nontrivial irreducible representation of $S_n$ corresponding to the partition $\lambda \vdash n$ and again let $\chi^\lambda$ be the corresponding character. %Clearly, $\rho_\lambda$  vanishes on the right hand-side side of equation \eqref{eq:group algebra} if $\rho_\lambda$ is not the trivial representation.
%Since $C$ is a conjugacy class, Schur's lemma gives
%\[
%\rho_\lambda\left(\sum_{c\in C}c\right)
% =\frac{|C|\chi^\lambda(c)}{\chi^\lambda(1)}I,
%\]
%where $I$ is the identity matrix.
%Therefore, it follows from equation \eqref{eq:group algebra} that whenever
As we have seen earlier, whenever $\chi^\lambda(c) \ne 0$, then $\rho_\lambda(\sum_{c\in \mathcal{C}}c)$ is invertible so we have
$\sum_{t\in T}\rho_\lambda(t)=0$.

Now let $H=S_k\times S_k$ be the Young subgroup of $S_n$.
%and let $1_H$ denote the trivial representation of $H$. Let us denote the induced representation of $1_H$ to $S_n$ by $1_H^{S_n}$.
By Young's rule,
\[
1_{H}^{S_n}
 =\sum_{\mu\vdash 2k}K_{\mu,(k^2)} \cdot \chi^\mu.
 \]
%where $K_{\mu,\lambda}$ denotes the Kostka number corresponding to this pair of partitions of $n$.
The Kostka number
$K_{\mu,\lambda}$ is nonzero if and only if $\mu$ dominates $\lambda$ in the dominance order of partitions which we define below.

%Let $p = (p_1,p_2,\ldots ,p_a)$ and $q = (q_1,q_2,\ldots q_b)$ be partitions of $n$. In this case $q$ dominates $p$, which is denoted by $p\trianglelefteq q$ if and only $ p_{1}+\cdots +p_{k}\leq q_{1}+\cdots +q_{k}$ for all $k\geq 1$.

It is easy to see that the partitions dominating $(k^2)$ are precisely the two-row partitions $(2k-i,i)$, $0\le i\le k$ since the first coordinate has to be at least $k$ and the sum of the first two is at least and then equal to $2k=n$.
We claim that every partition $(2k-i,i)\vdash n=2k$ has empty $2$-core. To see this, let us first note that its two row lengths have the same parity. If they are both even, both rows can be tiled with horizontal tiles only.
If both row lengths are odd, we may remove horizontal dominoes from the right ends of the two rows until only the first box of each row remains. These two boxes form a vertical domino, which can then be removed last.
 Thus $\chi^{(2k-i,i)}(c) \ne 0$ for $c\in C$ and for every $i$.

The induced representation of the trivial representation of a subgroup $H$ is the permutation representation of $S_n$ on the cosets of $H$. Let $P_g$ denote the corresponding permutation matrix, which is an element of $\mathbb{C}^{\frac{|G|}{|H|} \times \frac{|G|}{|H|}}$. The preceding vanishing statement shows that $\sum_{t\in T}P_t$ is the zero matrix on every nontrivial irreducible constituent of this permutation representation. On the one-dimensional trivial constituent it acts as multiplication by $|T|$. Therefore
\[
\sum_{t\in T}P_t
 =\frac{|T|}{[S_n:H]}J,
\]
where $J$ is the all-ones matrix. Since the left-hand side has integer entries, we must have that
$\frac{|T|}{[S_n:H]}$ is a positive integer.
Now
\[
|C|=\frac{(2k)!}{2^k k!} \mbox{ and then}
\qquad
|T|=\frac{(2k)!}{|C|}=2^k k!
\qquad \text{and }
[S_n:H]=\binom{2k}{k}.
\]
Hence $\binom{2k}{k}$ must divide $2^k k!$. By Bertrand's postulate, there is a prime $p$ with $k<p<2k$. This prime divides $\binom{2k}{k}$, but it divides neither $k!$ nor $2^k$, a contradiction.
\end{proof}

\subsection{Excluding involutions with one fixed point}\label{subsection:one-fixed-point}
Now we combine our elementary technique with some ideas from representation theory.
\begin{theorem}\label{thm:(2^k,1^1)}
Let $n=2k+1 \ge 5$, where $k$ is odd, and let $C=\mathcal{C}_{(2^k,1)}$. Then $C$ is not a tile in $S_n$.
\end{theorem}
\begin{proof}
    Let $D=\mathcal{C}_{(2^{k-2},4,1)}$. The main ingredient from representation theory is the following observation. Let us assume $(C,T)$ is a tiling for some $T \subset S_n$. We claim that in this case $(D,T)$ is a $\frac{|D|}{|C|}$-fold tiling.

        Since $C$ and $D$ are both conjugacy classes, their class sum act as scalar matrices in every irreducible representation. Thus, it is enough to verify that whenever $\chi(C)=0$, then $\chi(D)=0$. We show this by contrapositive, so let $\lambda \vdash n$ such that $\chi^{\lambda}(d) \ne 0$ for $d\in D$. It follows from the Murnaghan-Nakayama rule that $BST(\lambda,(2^{k-2},4,1))$ is nonempty.
We now show that $ BST(\lambda,(2^k,1))\neq\varnothing$.

We have already established that
$BST(\lambda,(2^{k-2},4,1))\ne \emptyset$, so it suffices to show
that the removal of a rim hook of length $4$ can be replaced by two
successive domino removals.

Let $R=\alpha \setminus \beta$ be the rim hook removed at the corresponding
stage, and list its cells as $b_1,b_2,b_3,b_4$ from northeast to
southwest, so that each successive step is either to the left or
downward. The order in which we remove the two dominoes
$\{b_1,b_2\}$ and $\{b_3,b_4\}$ is determined by the direction of
the middle step, from $b_2$ to $b_3$.

Suppose first that this middle step is downward, with $b_2$ in
row $r$ and $b_3$ in row $r+1$. Then $b_3,b_4$ are precisely the
cells of $R$ below row $r$. Removing these two cells from $\alpha$
therefore gives a diagram with row lengths
\[
\nu_i=
\begin{cases}
\alpha_i, &\mbox{if }~ i\le r,\\
\beta_i, & \mbox{if }~i>r.
\end{cases}
\]
These row lengths are weakly decreasing within each of the two
ranges, since $\alpha$ and $\beta$ are partitions. At the boundary
between the ranges, we have
\[
\nu_r=\alpha_r\ge\beta_r\ge\beta_{r+1}=\nu_{r+1}.
\]
Thus $\nu$ is also a partition. Consequently,
$\alpha/\nu=\{b_3,b_4\}$ is a removable domino, and
then $\nu/\beta=\{b_1,b_2\}$ is a removable domino in the resulting
diagram.

Notice that a diagram represents a partition if and only if its transpose is a partition.
If the middle step is to the left, we may apply the preceding argument
to the transposed diagrams. The cells of the transposed rim hook,
listed from northeast to southwest, are
$b_4^{\mathsf t},\ b_3^{\mathsf t},\ b_2^{\mathsf t},\ b_1^{\mathsf t},
$
and its middle step is downward. By the preceding case, we can first remove
$\{b_1^{\mathsf t},b_2^{\mathsf t}\}$ from the transposed diagram
and then $\{b_3^{\mathsf t},b_4^{\mathsf t}\}$, with a partition
remaining after each removal. Transposing back, we see that
$\{b_1,b_2\}$ can be removed first, followed by $\{b_3,b_4\}$,
again leaving a partition after each removal.

In either case, replacing the original rim-hook removal by these
two domino removals, while keeping all other steps unchanged,
gives an element of $BST(\lambda,(2^k,1))$, as required.
Thus
\[
BST(\lambda,(2^k,1))\ne\emptyset,
\]
and hence, by Lemma~\ref{lem:2-core--BST},
\[
\chi^\lambda(c)\ne0.
\]

As a result of this argument we indeed obtain that whenever $\chi^{\lambda}(d) \ne 0$ for some $d\in D$, then $\chi^{\lambda}(c) \ne 0$ which proves our claim.

 Let $A=\mathcal{C}_{(2^{k-1},3^1)}$ and $B=\mathcal{C}_{(2^{k-1},1^3)}$. The fact that $(C,T)$ is a tiling implies that
 $$\underline{C} \cdot \underline{T}=\underline{S_n}.
 $$
 Let us now multiply this equation from the left by $\underline{X}$, where $X=\mathcal{C}_{(2^1,1^{2k-1})}$.
 Sum of normal sets are contained in the center of the group algebra so we can write it as
$$(\underline{C} \cdot \underline{X}) \cdot
\underline{T}=  \underline{X} \cdot\underline{S_n} =|X| \underline{S_n}.
 $$
By Lemma \ref{lem:cycle structure}, $CX$ as a set is equal to $A\cup B\cup D$.
Easy calculations show that $\underline{C} \cdot \underline{X}=3\underline{A}+3 \underline{B}+2\underline{D}$ so if we multiply this by $\underline{T}$, then using  $\underline{D} \cdot \underline{T}=\frac{|D|}{|C|} \underline{S_n}=k(k-1)\underline{S_n}$ we obtain
\begin{equation*}\label{eq:groupalgebra1}
3(\underline{A}+\underline{B})\cdot \underline{T}+2 k (k-1) \underline{S_n}=\binom{2k+1}{2}\underline{S_n}.
\end{equation*}
This gives $$(\underline{A}+\underline{B})\underline{T}=k\underline{S_n}, $$
which means that $(A\cup B,T)$ is a $k$-fold tiling of $S_n$. In particular, we obtain the fact that $|A|+|B|=k|C|$ which is of course easy to check.

Again, we may assume $1\in T$ and since $(A\cup B,T)$ is a $k$-fold tiling of $S_n$ we have $1=st$ for some $t \in T$ and $s\in A$ or $s \in B$. Thus $T^{-1}T$ contains an element of $A$ or $B$.
Sands's criterion implies that $C^{-1}C \cap T^{-1}T =\{1\}$ since $(C,T)$ is a tiling. Thus in order to obtain a contradiction, it is enough to prove $A \cup B \subset C^{-1}C=C^2$. The cycle structure of all permutations in $A$ and $B$ contains only 1 or 3 odd cycles and the rest of the cycles are transpositions. Since $A$ and $B$ are contained in $A_n$, their elements contain an even number of transpositions. Thus it follows from Lemma~\ref{lem:ket konjugalt osztaly negyezete}(\ref{item:U negyzete}) that $A \subset C^2$ and $B \subset C^2$. Elementary calculations from earlier parts of this paper also show these inclusions, and hence we obtain a contradiction.
\end{proof}
\section{Concluding remarks}\label{subsection: limitation of techniques}
The main aim of this section is to explain why our elementary method does not handle the two cases treated in Subsections~\ref{subsection:fixed-point-free-involutions} and~\ref{subsection:one-fixed-point}, and why representation theory is needed. We also discuss possible extensions and prove Lemma~\ref{lem:ket konjugalt osztaly negyezete}, which was used in the preceding section.

Let $n=2k$ with $k$ odd, and let $S$ be the conjugacy class of permutations of cycle type $(2^k)$.
The main result of this subsection is Proposition~\ref{prop:NSGSCP}. It shows that a different technique is needed to prove non-tiling for this family of conjugacy classes. The situation is similar for the conjugacy class of cycle type $(2^k,1)$ when $k$ is odd.
\begin{prop}\label{prop:NSGSCP}
If $Sg \subset S^2$, then $g \in S$ or $g=1$.
\end{prop}
To prove Proposition~\ref{prop:NSGSCP}, we first describe the elements of $S^2$.
\begin{lemma}\label{lem:ket konjugalt osztaly negyezete}
\begin{enumerate}[\rm(a)]
    \item\label{item:S negyzete} Let $S=\mathcal{C}_{(2^k)} \subset S_{2k}$. Then $g\in S^2$ if and only if the number of $\ell$-cycles in $g$ is even for every $ \ell \le 2k$.
\item\label{item:U negyzete}
Let $U=\mathcal{C}_{(2^k,1^1)} \subset S_{2k+1}$. Then   $g\in U^2$ if and only if the number of
$\ell$-cycles in
$g$ is even for every $\ell \le 2k+1$ except for exactly one odd value of $\ell$.
\end{enumerate}
\end{lemma}
\begin{proof}
   \textbf{(a)} Take two elements $s_1,s_2$ of $S$. We may assign a matching for both permutations on the set $\{1,2,\ldots, 2k\}$. The edges corresponding to $s_1$ are coloured red and the others blue.  The graph obtained in this way is called $\Gamma$. Every point in this graph has degree 2 so it is the union of cycles. The length of each cycle is even since the colours along each cycle alternate.

The product $s_1s_2$ can be read from these alternating cycles. Suppose that one component has vertices $1,2,\ldots,2\ell$ in cyclic order, with $\{1,2\}$ red and $\{2,3\}$ blue. Then the restriction of $s_1s_2$ to this component is
\[
(1,3,\ldots,2\ell-1)(2\ell,2\ell-2, \ldots, 2).
\]
Thus every cycle length occurs an even number of times in $s_1s_2$.

Conversely, if every cycle length occurs an even number of times in $g$, pair the cycles of each length. For each pair of $\ell$-cycles, place their $2\ell$ points on one alternating red--blue cycle so that the construction above yields the prescribed pair. The union of the red edges and the union of the blue edges are perfect matchings, and their product is $g$.

\textbf{(b)}
Let us now start with a pair of elements $u_1,u_2$ of $U$.
We may build up a similar graph whose vertices are indexed by $[2k+1]$ and we draw a red edge in the graph for each transposition in $u_1$ and draw blue edges for the transpositions of $u_2$ and we draw coloured loops for both fixed points. Then the product $u_1u_2$ again correponds to paths of lengts 2 whose first edge is red and the second one is blue.

If the two fixed points coincide then our statement follows from (\ref{item:S negyzete}).
Otherwise, we may assume $1^{g_1}=1$ and $1^{g_2} \ne 1$. Now every point in the graph has degree 2 except for those being attached to a loop. So a cycle staring at 1 must coincide with a path ending in the other fixed point. It is clear that this path will provide a cycle in $g_1g_2$.
Now if we delete the vertices of the path, we obtain a subgraph we have seen in (\ref{item:S negyzete}) so the restriction of the product to that subset has all cycles with even multiplicity. Thus the path we first deleted has to contain odd vertices, which proves one direction of our second statement.

The reverse direction, the reconstruction of $g_1$ and $g_2$ is again clear from the argument above.
\end{proof}
%Similarly, one can show that the elements of $\left( \mathcal{C}_{(2^{2a+1},1^1)} \right)^2$
 % consists of the permutations in which every cycle length appears an even number of times, except for exactly one odd number.
\begin{proof}[Proof of Proposition~\ref{prop:NSGSCP}]
Assume that $Sg\subseteq S^2$. Since $k$ is odd, every element of $S$ is odd, whereas every element of $S^2$ is even. Hence $g$ must be odd. We first show that $g$ has no cycle of length at least $5$. We construct two elements of $Sg$ whose numbers of fixed points differ by exactly one; consequently, one of them has an odd number of fixed points and cannot belong to $S^2$. Let us now fix a cycle $(1,2,\ldots,\ell) $ of length $\ell\ge 5$. View elements of $S$ as perfect matchings on $[n]:=\{1,2,\ldots, n\}$. Fix the matching outside the displayed cycle, leaving one outside point $x$ available when $\ell$ is odd.
We extend this partial matching in two ways. A matching edge joining two distinct cycles of $g$ cannot contribute a fixed point to the product.

Let us assume first that
$\ell=2m+2$, $m \ge 2$.
\begin{itemize}
    \item If $s\vert_{[\ell]}=(1,2)(3,3+m)(4,4+m) \ldots (\ell-m,\ell)$, then $sg$ has one fixed point on $[\ell]$.
    \item If $s\vert_{[\ell]}=(1,m+2)(2,m+3)\ldots(m+1,\ell)$, then $sg$ has no fixed point on $[\ell]$.
\end{itemize}
Let us assume now that
$\ell=2m+3$, $m \ge 1$.
\begin{itemize}
\item $\ell=5$: If $s\vert_{[5]}=(1,2)(3,5)(4,x)$, where $x\in [n]\setminus [5]$, then $sg$ has one fixed point on $[5]$.
    \item $\ell=2m+3$, $m\ge 2$:
If $s\vert_{[\ell]}=(1,2)(3,x)(4,4+m)(5,5+m)\ldots (\ell-m,\ell)$, where $x\in [n]\setminus [\ell]$, then $sg$ has one fixed point on $[\ell]$.
\item If $s\vert_{[\ell]}=(1,m+2)(2,m+3) \ldots (\ell-m-2,\ell-1)(\ell,x)$, where $x\in [n]\setminus [\ell]$, then $sg$ has no fixed point on $[\ell]$.
\end{itemize}
It follows from these calculations that we may assume that the longest cycle of $g$ has length at most $4$.

Now we exclude the presence of two 3-cycles in the cycle structure of $g$.
Suppose, for a contradiction, that $g=(1,2,3)(4,5,6) g_1$ with $\supp(g_1)\cap[6]=\emptyset$.
\begin{itemize}
    \item If $s\vert_{[6]}=(1,2)(3,4)(5,6)$, then $gs\vert_{[6]}=(1)(5)(2,4,6,3)$
    \item while if $s\vert_{[6]}=(1,4)(2,5)(3,6)$, then $gs\vert_{[6]}=(1,5,3,4,2,6)$.
\end{itemize}
Keeping the matching fixed on the remaining points, the two products differ by one in the parity of the number of $4$-cycles and also by one in the parity of the number of $6$-cycles. Hence at least one of them does not belong to $S^2$. Thus $g$ has at most one $3$-cycle. \newline
\textbf{4-cycles:}
Let us assume now that $g$ has two 4-cycles: $(1,2,3,4)(5,6,7,8)$.
\begin{itemize}
\item $(1,2,3,4)(5,6,7,8)(1,2)(3,4)(5,6)(7,8)$ has four fixed points and
\item $(1,2,3,4)(5,6,7,8)(1,5)(2,4)(3,6)(7,8)$ has one fixed point.
\end{itemize}
Let us assume now that $g$ has a 4-cycle and a 3-cycle: $(1,2,3,4)(5,6,7)$.
\begin{itemize}
\item $(1,2,3,4)(5,6,7)(1,5)(2,6)(3,7)(4,x)$ has no fixed points and
\item $(1,2,3,4)(5,6,7)(1,2)(3,5)(4,6)(7,x)$ has one fixed point.
\end{itemize}
These calculations show that if $g$ has a $4$-cycle, then every other cycle of $g$ has length at most $2$. Because $g$ is odd, the number of transpositions in its cycle structure is even. Since $n\equiv2\pmod4$, it follows that $g$ has at least two fixed points.
This case is excluded by the following two products:
\begin{itemize}
\item $(1) (2) (3,4,5,6) (1,2)(3,5)(4,6)$ has no fixed points and
\item $(1) (2) (3,4,5,6) (1,5)(2,6)(3,4)$ has one fixed point.
\end{itemize}
\textbf{3-cycles:} It remains to handle the case when the longest cycle in $g$ is a 3-cycle. There is at most one such cycle. Since $g$ is odd and $n$ is even, the remaining cycle structure contains both a transposition and a fixed point; otherwise $g$ would be even.
Let $g\vert_{[6]}=(1,2,3)(4,5)(6) $.
\begin{itemize}
\item $(1,2,3)(4,5)(6) \cdot (1,4)(2,5)(3,6)$ has no fixed points and
\item $(1,2,3)(4,5)(6) \cdot (1,2)(3,4)(5,6)$ has one fixed point.
\end{itemize}
Finally, suppose that $g$ is an involution. If $g\notin S$, then it has fixed points. Since $g$ is odd and $n=2k$ with $k$ odd, it must have at least four fixed points. On six suitable points, write $g=(1,2)(3)(4)(5)(6)$. Then
\[
(1,2)(3,4)(5,6)g=(1)(2)(3,4)(5,6),
\]
whereas
\[
(1,3)(2,4)(5,6)g=(1,3,2,4)(5,6).
\]
The two products differ by one in the parity of the number of $4$-cycles. After fixing the matching on the remaining points, at least one of them therefore lies outside $S^2$. This contradiction proves that $g\in S$.
\end{proof}

We conclude with the following natural general problem.

\begin{problem}\label{prob:finite-groups}
Given a finite group $G$, determine all nonidentity conjugacy classes of $G$ that tile $G$.
\end{problem}

The groups considered in this paper suggest that almost simple groups form a natural family for which one might hope to obtain a complete, or nearly complete, solution. For $G=\operatorname{PSL}(2,q)$, the results of Garion \cite{garion} and Terada \cite{terada}, together with Lemma~\ref{Lem:NecessaryCondition}, appear to give such a solution. We leave the details for future work.

\section*{Acknowledgement}
 This work was supported by a Discovery Project grant (DP250104965) of the Australian Research Council.


\begin{thebibliography}{00}
\bibliographystyle{plain}

\bibitem{cameron}
Cameron, P. J., Yap, R. S., \& Zhou, S. (2026). Perfect codes in Cayley graphs of abelian groups. \textit{Designs, Codes and Cryptography, 94}(4), Article 87.

\bibitem{FangLi}
Fang, T., \& Li, J. (2026). Proof of a conjecture of Green and Liebeck on codes in symmetric groups. Discrete Mathematics, 349(5), 114999.

\bibitem{transpositions}
Fang, T., \& Xia, B. (2026). Tiling the symmetric group by transpositions. \textit{Bulletin of the London Mathematical Society, 58}(5), e70366.

\bibitem{garion}
Garion, S. (2015). Expansion of conjugacy classes in $\operatorname{PSL}_2(q)$. \textit{Journal of Group Theory, 18}(6), 961--980.

\bibitem{perfectcodes}
Huang, H., Xia, B., \& Zhou, S. (2018). Perfect codes in Cayley graphs. \textit{SIAM Journal on Discrete Mathematics, 32}(1), 548--559.

\bibitem{Gleason}
Husemoller, D. H. (1962). Ramified coverings of Riemann surfaces. \textit{Duke Mathematical Journal, 29}, 167--174.

\bibitem{GreenLiebeck}
Green, H. M., \& Liebeck, M. W. (2020). Some codes in symmetric and linear groups. Discrete Mathematics, 343(8), 111719.


\bibitem{James}
James, G. D. (1978). \textit{The representation theory of the symmetric groups} (Lecture Notes in Mathematics, Vol.~682). Springer.

\bibitem{Kourovka}
E. I. Khukhro and V. D. Mazurov (eds.),
\emph{The Kourovka Notebook: Unsolved Problems in Group Theory},
20th ed., Institute of Mathematics, Siberian Branch of the Russian Academy of Sciences,
Novosibirsk, 2022, Problem 9.24, p.~33.

\bibitem{kol}
Kolountzakis, M. N. (2004). The study of translational tiling with Fourier analysis. In L. Brandolini, L. Colzani, A. Iosevich, \& G. Travaglini (Eds.), \textit{Fourier analysis and convexity} (pp.~131--187). Birkh"auser.

\bibitem{labalondner}
\L aba, I., \& Londner, I. (2025). Splitting for integer tilings. \textit{International Mathematics Research Notices, 2025}(8), rnaf090.

\bibitem{LTuniform}
Larsen, M., \& Tiep, P. H. (2024).
Uniform character bounds for finite classical groups.
\emph{Annals of Mathematics, 200}(1), 1--70.

\bibitem{LTasymptotic}
Larsen, M., \& Tiep, P. H. (2025).
Character estimates for finite classical groups and the asymptotic
Thompson Conjecture.
\emph{arXiv preprint arXiv:2403.09047v2}.


\bibitem{LT2023}
Larsen, M. J., \& Tiep, P. H. (2023). Squares of conjugacy classes in alternating groups. arXiv preprint arXiv:2305.04806.
\bibitem{LWZ2026}
Liao, J., Wang, L., \& Zhang, J. (2026).
Thompson's conjecture for finite simple groups of Lie type over
small fields.
\emph{arXiv preprint arXiv:2608.17891}.
.

\bibitem{subgroupperfect}
Li, B., Li, J., Meng, W., \& Yu, H. (2026). Characterizing finite groups via subgroup perfect codes. \textit{arXiv preprint arXiv:2605.03284}.

\bibitem{levymaroti}
Levy, D., \& Mar\'{o}ti, A. (2018). Set-direct factorizations of groups. \textit{Journal of Algebra, 516}, 414--436.

\bibitem{rothausthompson}
Rothaus, O. S., \& Thompson, J. G. (1966). A combinatorial problem in the symmetric group. \textit{Pacific Journal of Mathematics, 18}(1), 175--178.

\bibitem{reprtheoryofsymmetric}
Sagan, B. E. (2001). \textit{The symmetric group: Representations, combinatorial algorithms, and symmetric functions} (2nd ed., Graduate Texts in Mathematics, Vol.~203). Springer.

\bibitem{sands}
Sands, A. D. (1957). On the factorisation of finite abelian groups. \textit{Acta Mathematica Academiae Scientiarum Hungaricae, 8}(1--2), 65--86.

\bibitem{terada}
Terada, S. (2004). Perfect codes in $\operatorname{SL}(2,2^f)$. \textit{European Journal of Combinatorics, 25}(7), 1077--1085.

\bibitem{totalperfectcodes}
Zhou, S. (2016). Total perfect codes in Cayley graphs. \textit{Designs, Codes and Cryptography, 81}(3), 489--504.

\end{thebibliography}
\end{document}